\documentclass{amsart}

\usepackage{graphicx,amssymb,mathrsfs,amsmath,color,fancyhdr}
\usepackage[all]{xy}
\newtheorem{theorem}{Theorem}[section]
\newtheorem{lemma}[theorem]{Lemma}

\theoremstyle{definition}
\newtheorem{definition}[theorem]{Definition}

\newtheorem{conjecture}[theorem]{Conjecture}

\theoremstyle{remark}
\newtheorem{remark}[theorem]{Remark}

\numberwithin{equation}{section}

\newcommand{\bdot}{\boldsymbol{\cdot}}

\begin{document}

\title[Erd\H{o}s-Ginzburg-Ziv theorem and Noether number for $C_m\ltimes_{\varphi} C_{mn}$]
{Erd\H{o}s-Ginzburg-Ziv theorem and Noether number for $C_m\ltimes_{\varphi} C_{mn}$}

\begin{abstract}
Let $G$ be a multiplicative finite group and $S=a_1\bdot\ldots\bdot a_k$ a sequence over $G$. We call $S$ a product-one sequence if $1=\prod_{i=1}^ka_{\tau(i)}$ holds for some permutation $\tau$ of $\{1,\ldots,k\}$. The small Davenport constant $\mathsf d(G)$ is the maximal length of a product-one free sequence over $G$. For a subset $L\subset \mathbb N$, let $\mathsf s_L(G)$ denote the smallest $l\in\mathbb N_0\cup\{\infty\}$ such that every sequence $S$ over $G$ of length $|S|\ge l$ has a product-one subsequence $T$ of length $|T|\in L$. Denote $\mathsf e(G)=\max\{\text{ord}(g): g\in G\}$. Some classical product-one (zero-sum) invariants including $\mathsf D(G):=\mathsf s_{\mathbb N}(G)$ (when $G$ is abelian), $\mathsf E(G):=\mathsf s_{\{|G|\}}(G)$, $\mathsf s(G):=\mathsf s_{\{\mathsf e(G)\}}(G)$, $\eta(G):=\mathsf s_{[1,\mathsf e(G)]}(G)$ and $\mathsf s_{d\mathbb N}(G)$ ($d\in\mathbb N$) have received a lot of studies. The Noether number $\beta(G)$ which is closely related to zero-sum theory is defined to be the maximal degree bound for the generators of the algebra of polynomial invariants. Let $G\cong C_m\ltimes_{\varphi} C_{mn}$, in this paper, we prove that
$$\mathsf E(G)=\mathsf d(G)+|G|=m^2n+m+mn-2$$ and $\beta(G)=\mathsf d(G)+1=m+mn-1$. We also prove that $\mathsf s_{mn\mathbb N}(G)=m+2mn-2$ and provide the upper bounds of $\eta(G)$, $\mathsf s(G)$. Moreover, if $G$ is a non-cyclic nilpotent group and $p$ is the smallest prime divisor of $|G|$, we prove that
$\beta(G)\le \frac{|G|}{p}+p-1$ except if $p=2$ and $G$ is a dicyclic group, in which case $\beta(G)=\frac{1}{2}|G|+2$.
\end{abstract}

\author{Dongchun Han}
\address{Department of Mathematics, Southwest Jiaotong University, Chengdu 610000, P.R. China}
\email{han-qingfeng@163.com}
\author{Hanbin Zhang}
\address{Academy of Mathematics and Systems Science, Chinese Academy of Sciences, Beijing
100190, P.R. China}
\email{nkuzhanghanbin@163.com}

\keywords{}
\maketitle

\section{Introduction}

Let $G$ be a multiplicative finite group. By a sequence over $G$, we mean a finite sequence of terms from $G$ which is unordered and repetition of terms allowed. We say that $S$ is a product-one sequence if its terms can be ordered so that their product equals 1, the identity of $G$. The small Davenport constant, denoted by $\mathsf d(G)$, is the maximal length of a product-one free sequence over $G$.
For a subset $L\subset \mathbb N$, let $\mathsf s_L(G)$ denote the smallest $l\in\mathbb N_0\cup\{\infty\}$ such that every sequence $S$ over $G$ of length $|S|\ge l$ has a product-one subsequence $T$ of length $|T|\in L$.
Denote $\mathsf e(G)=\max\{\text{ord}(g): g\in G\}$. Some classical examples of product-one (zero-sum) invariants including $\mathsf D(G):=\mathsf s_{\mathbb N}(G)$ (when $G$ is abelian), $\mathsf E(G):=\mathsf s_{\{|G|\}}(G)$, $\mathsf s(G):=\mathsf s_{\{\mathsf e(G)\}}(G)$, $\eta(G):=\mathsf s_{[1,\mathsf e(G)]}(G)$ and $\mathsf s_{d\mathbb N}(G)$ ($d\in\mathbb N$) have received a lot of studies, see \cite{GG} for a survey.

In 1961, Erd\H{o}s, Ginzburg and Ziv {\rm\cite{EGZ}} showed that $\mathsf E(G)\leq 2|G|-1$ for every finite solvable group $G$ and which implies that $\mathsf E(G)=\mathsf d(G)+|G|=2|G|-1$ for every finite cyclic group $G$. This result is well known as the Erd\H{o}s-Ginzburg-Ziv theorem. In 1984, Yuster and Peterson \cite{YP} showed that when $G$ is a non-cyclic solvable group, then $\mathsf E(G)\le 2|G|-2$. Later Yuster \cite{Y} improved the result to $\mathsf E(G)\le 2|G|-r$ provided that $|G|\ge 600((r-1)!)^2$. In 1996, Gao \cite{Gao2} improved the bound to $\mathsf E(G)\le \frac{11|G|}{6}-1$. Later in 2009 Gao and Li \cite{GLi} proved that $\mathsf E(G)\le \frac{7|G|}{4}-1$ and they conjectured that $\mathsf E(G)\le \frac{3|G|}{2}-1$ for any finite non-cyclic group.

When $G$ is abelian, Gao \cite{Gao1} proved the fundamental relation
$$\mathsf E(G)=\mathsf d(G)+|G|.$$

For a weighted version of this formula, we refer to the Chapter 16 of \cite{Gry}. Later, Zhuang and Gao \cite{ZG} proposed the following conjecture.

\begin{conjecture}\label{con1}
For every finite group $G$, $\mathsf E(G)=\mathsf d(G)+|G|$.
\end{conjecture}

Conjecture \ref{con1} attracts a lot of attention. Zhuang and Gao {\rm\cite{ZG}} verified the conjecture for dihedral groups of order $2p$ where $p\ge4001$ is a prime. Gao and Lu {\rm\cite{GLu}} improved the result to dihedral groups of order $2n$ for all $n\ge23$. Let $C_m\ltimes_{\varphi} C_n$ denote any semidirect product of a normal cyclic subgroup of order $n$ and a subgroup of order $m$, with any $\varphi:C_m\rightarrow$ Aut$(C_n)$ being a group homomorphism. J. Bass {\rm\cite{Bass}} extended the method of Gao and Lu to prove the conjecture for all dihedral groups, dicyclic groups and $C_p\ltimes_{\varphi} C_q$, where $p$, $q$ are primes. The first author \cite{Han} verified the conjecture when $G\cong C_p\ltimes_{\varphi} C_{pn}$, where $p$ is a prime, $n$ is a positive integer. He also verified the above conjecture of Gao and Li for non-cyclic nilpotent groups.

In this paper, we prove the following result.
\begin{theorem} \label{theorem1}
Let $G$ be a finite group, and $m$ be any positive integer. If $G$ has a normal subgroup $N$ such that $G/N\cong C_m\ltimes_{\varphi} C_m$, then
$$|G|+\mathsf d(G)\le \mathsf E(G)\le |G|+\frac{|G|}{m}+m-2.$$
In particular, if $G\cong C_m\ltimes_{\varphi} C_{mn}$, where $m,n$ are positive integers, then
$$\mathsf E(G)=|G|+\mathsf d(G)=m^2n+mn+m-2.$$
As a consequence, we have $$\mathsf d(G)=mn+m-2.$$
\end{theorem}

Next, we are going to investigate $\mathsf s(G)$, $\eta(G)$ and $\mathsf s_{d\mathbb N}(G)$ ($d\in\mathbb N$).
Note that if $G$ is nilpotent, then $G$ is the direct product of its $p$-Sylow subgroups and hence $\mathsf e(G)=\text{lcm}\{\text{ord}(g):g\in G\}$. If $G$ is abelian, then $\mathsf e(G)=\exp(G)$ which is called the exponent of $G$. When $G$ is abelian, $\eta(G)$ and $\mathsf s(G)$ have received a lot of studies since the 1960s. In \cite{Gao3}, Gao conjectured that $\mathsf s(G)=\eta(G)+\mathsf e(G)-1$ holds for all abelian groups and he verified this conjecture for all group with $\exp(G)\le 4$. For some main results on $\mathsf s(G)$ and $\eta(G)$, see \cite[Sections 5.7 and 5.8]{GH}  and \cite{GHZ} for a recent progress. For non-abelian groups we refer to \cite[Sections 2.5 and 3.3]{CDG}. The study of $\mathsf s_{d\mathbb N}(G)$ ($d\in\mathbb N$) was proposed in \cite{GGS} and some results about abelian groups were obtained.

In this paper, we study $\eta(G)$, $\mathsf s(G)$ and $\mathsf s_{mn\mathbb N}(G)$ for the group $C_m\ltimes_{\varphi} C_{mn}$ and prove the following result.

\begin{theorem}\label{theorem2}
Let $G\cong C_m\ltimes_{\varphi} C_{mn}$. We have
$$\mathsf s_{mn\mathbb N}(G)=m+2mn-2.$$
If $\mathsf e(G)=mn$, then we have
$$\eta(G)\le 2m+mn-2\text{ and }\mathsf s(G)\le 2m+2mn-3.$$
\end{theorem}

Our next topic is the Noether number in invariant theory, which had been shown to be greatly connected with the zero-sum theory in recent years. Recall that the Noether number $\beta(G)$ of a finite group $G$ is $\sup_V\beta(G,V)$, where $V$ ranges over all finite dimensional $G$-modules $V$ over a fixed base field $\mathbb F$, and $\beta(G,V)$ is the smallest integer $d$ such that the algebra $\mathbb{F}[V]^G=\{f\in\mathbb{F}[V]|\text{ }f^g=f,\text{ for all }g\in G\}$ of polynomial invariants is generated by its elements of degree at most $d$.

In 1916, E. Noether \cite{Noe} proved that $\beta(G)\le |G|$ provided that char$(\mathbb{F})=0$. It can be easily verified that $\beta(C_n)=n$, where $C_n$ is a cyclic group of order $n$. B.J. Schmid \cite{BS} proved that for non-cyclic groups, Noether's bound was never sharp, that is $\beta(G)\le |G|-1$ for non-cyclic group $G$. Meanwhile, she showed that for any abelian group $G$, $\beta(G)=\mathsf D(G)=\mathsf d(G)+1$, which is very interesting and established a connection between invariant theory and zero-sum theory. Moreover, she also showed that the key step to improving the Noether bound is to find a better upper bound for $\beta(C_p\ltimes_{\varphi} C_q)$, where $C_p\ltimes_{\varphi} C_q$ is the semidirect product of cyclic groups of odd prime order. For the history of this problem and recent progress, see the recent paper \cite{CDG} by Cziszter, Domokos and Geroldinger (and the references there), their paper contains a wonderful survey in this topic, and in Section 5 of their paper, they showed a lot of striking similarities of features of the Noether number and the Davenport constant. Also see \cite{CD4, CDS, HMP} for the very recent progress on Noether number and Davenport constant.

In this paper, we obtain the precise value of the Noether number for $C_m\ltimes_{\varphi} C_{mn}$.

\begin{theorem} \label{theorem3}
Let $m,n$ be any positive integers, then
$$\beta(C_m\ltimes_{\varphi} C_{mn})=mn+m-1.$$
\end{theorem}

The result shows that $\beta(C_m\ltimes_{\varphi} C_{mn})=\mathsf d(C_m\ltimes_{\varphi} C_{mn})+1$ which is consistent with the abelian group case. Moreover, we give an upper bound for the nilpotent group.

\begin{theorem} \label{theorem4}
Let $G$ be a non-cyclic nilpotent group and $p$ the smallest prime divisor of $|G|$, then
$$\beta(G)\le \frac{|G|}{p}+p-1$$
except if $p=2$ and $G$ is a dicyclic group, in which case $\beta(G)=\frac{1}{2}|G|+2$.
\end{theorem}

Compared with the existing upper bound \cite{CD2, CD3, DH}, our result is an improvement when $G$ is a non-cyclic nilpotent group if $p$ is large.

\section{Preliminaries}

This section will provide more rigorous definitions for the above concepts and introduce notations that will be used repeatedly below.

Throughout, let $G$ be a finite group written multiplicatively. We define a $sequence$ over $G$ to be an element of the free abelian monoid $\big(\mathcal F(G),\bdot\big)$, see Chapter 5 of \cite{GH}, Section 3.1 of \cite{CDG} or \cite{GG} for detailed explanation. Our notation of sequences follows the notation in the papers \cite{GeG, Gry1, Oh}. Note that, since $G$ is a multiplicative group, we denote by $\emptyset$ the unit element of $\mathcal F(G)$ which is called the empty sequence. In particular, in order to avoid confusion between exponentiation of the group operation in $G$ and exponentiation of the sequence operation $\bdot$ in $\mathcal F (G)$, we define:
\[
g^{[k]}=\underset{k}{\underbrace{g\bdot\ldots\bdot g}}\in \mathcal F (G)\quad \text{and} \quad T^{[k]}=\underset{k}{\underbrace{T\bdot\ldots\bdot T}}\in \mathcal F (G) \,,
\]
for $g \in G$, \ $T\in \mathcal F (G)$ and $k \in \mathbb N_0$.

Let $S=g_1\bdot\ldots\bdot g_l=\prod_{g\in G_0}g^{\mathsf v_g(S)}$ be a sequence over $G$. Then the $length$ of $S$ is $|S|=l$. We use
$$\pi(S)=\{g_{\tau(1)}\cdots g_{\tau(l)}\in G: \tau\text{ a permutation of }[1,l]\}\subset G$$
to denote the set of products of $S$ (if $|S|=0$, we use the convention that $\pi(S)=\{1_G\}$), in particular, let $\sigma(S)=g_1\cdots g_l$.

We define
$$\prod(S)=\bigcup_{T|S,\emptyset\neq T}\pi(T)\subset G$$
to be the subsequence products of $S$. Moreover, for $1\le k\le |S|-1$, we define
$$\prod_k(S)=\bigcup_{T|S,\text{ }\emptyset\neq T\text{, }|T|=k}\pi(T)\subset G.$$
With these notations, a sequence $S$ is called
\begin{itemize}

\item a product-one sequence if $1_G\in\pi(S)$,

\item product-one free if $1_G\notin\prod(S)$.

\item minimal product-one sequence if $1_G\in\pi(S)$ and $S$ cannot be factored as a product of two non-trivial product-one subsequences.

\end{itemize}

For recent study on the algebraic and arithmetic structure of product-one sequence for non-abelian groups, we refer to \cite{Oh}. Using the above concepts, let

\begin{itemize}
\item the small Davenport constant $\mathsf d(G)$ denote the maximal length of a product-one free sequence over $G$.

\item the large Davenport constant $\mathsf D(G)$ denote the  maximal length of a minimal product-one sequence over $G$.

\end{itemize}

In this paper, we will deal with the small Davenport constant, for recent progress on large Davenport constant, we refer to \cite{GeG} and \cite{Gry1}.

For a subset $L\subset \mathbb N$, let $\mathsf s_L(G)$ denote the smallest $l\in\mathbb N_0\cup\{\infty\}$ such that every sequence $S$ over $G$ of length $|S|\ge l$ has a product-one subsequence $T$ of length $|T|\in L$. Let $\mathsf e(G)=\max\{\text{ord}(g): g\in G\}$. With these notation, we can introduce some classical examples of product-one (zero-sum) invariants:
$\mathsf D(G):=\mathsf s_{\mathbb N}(G)$ (when $G$ is abelian), $\mathsf E(G):=\mathsf s_{\{|G|\}}(G)$, $\mathsf s(G):=\mathsf s_{\{\mathsf e(G)\}}(G)$, $\eta(G):=\mathsf s_{[1,\mathsf e(G)]}(G)$ and $\mathsf s_{d\mathbb N}(G)$ ($d\in\mathbb N$).

Note that, if $G$ is nilpotent, then $G$ is the direct product of its $p$-Sylow subgroups and hence $\mathsf e(G)=\text{lcm}\{\text{ord}(g):g\in G\}$. If $G$ is abelian, then $\mathsf e(G)=\exp(G)$ which is called the $exponent$ of $G$.

The next lemma formulates a basic relationship between $\eta(G)$ and $\mathsf s(G)$. Although the proof runs along the same lines as in the abelian groups, we provide it in full detail.

\medskip
\begin{lemma}\label{lemma0}
Let $G$ be a finite group. Then $\eta (G)\le \mathsf s(G)-\mathsf e(G)+1$.
\end{lemma}

\begin{proof}
Let $S \in \mathcal F (G)$ be a sequence of length $|S| \ge \mathsf s (G) - \mathsf e (G)+1$. We have to verify that $S$ has product-one subsequence of length in $[1, \mathsf e (G)]$. The sequence $T = S \bdot 1^{[\mathsf e (G)-1]}$ satisfies $|T| \ge \mathsf s (G)$ and thus there exists a product-one subsequence $T' = S' \bdot 1^{[k]}$ where $k \in [0, \mathsf e (G)-1]$, $S' \mid S$, and $|T'|=|S'|+k=\mathsf e (G)$. Thus $S'$ is a product-one subsequence of $S$ with length $|S'| \in [1, \mathsf e (G)]$.
\end{proof}

Next, we recall \cite{Group} the definition of $C_m\ltimes_{\varphi} C_n$. It is generated by two elements $x,y$, where ord$(y)=m$ and ord$(x)=n$, $\langle x\rangle\cap\langle y\rangle={1}$ and $\varphi:C_m\rightarrow$  Aut$(C_n)$ being a group homomorphism such that $\varphi(y)\cdot x=yxy^{-1}=x^s$, $1\leq s\leq n-1$. It can be easily verified that $s^m\equiv1\pmod{n}$.

We employ the following lemmas in our proof.

\begin{lemma}\label{lemma1}
For every finite group $G$, $\mathsf d(G)+|G|\le\mathsf E(G)\le 2|G|-1.$
\end{lemma}

\begin{proof}
The lower bound can be found in \cite{ZG}, and the upper bound can be found in \cite{Olson1}.
\end{proof}

\begin{lemma} \label{lemma2}~
\begin{enumerate}
\item If $G\cong C_p^r$ for a prime $p$ and $r \in \mathbb N$, then $\mathsf D (G) = r(p-1)+1$.

\smallskip
\item If $G\cong C_n$, then $\mathsf D(G)=\eta(G)=n$ and $\mathsf s(G) = 2n-1$.

\smallskip
\item If $G\cong C_{n_1} \oplus C_{n_2}$ with $1 \le n_1 \mid n_2$, then $\mathsf d (G)=n_1+n_2-2$, $\eta (G) = 2n_1+n_2 - 2$, and $\mathsf s (G) = 2n_1+2n_2-3$.

\smallskip
\item If $G\cong C_{n_1} \oplus C_{n_2}$ with $1 \le n_1 \mid n_2$, then $\mathsf s_{n_2\mathbb N}(G)=n_1+2n_2-2$.
\end{enumerate}
\end{lemma}

\begin{proof}
The proofs of (1), (2) and (3) can be found in Chapter 5 of \cite{GH}, see Theorem 5.5.9, Corollary 5.7.5 and Theorem 5.8.3. The proof of (4) can be found in \cite{GGS}, see Theorem 5.2.
\end{proof}

\begin{lemma} \label{lemma4}{\rm(\cite{Han}, Lemma 2.4)}
Let $S$ be a sequence over $C_n$.
\begin{enumerate}

\item If $|S|=kn+n-1$, then $S$ contains a product-one subsequence $T$ of length $kn$;

\item If $|S|=kn+n-2$ and $S$ contains no product-one subsequence of length $kn$, then $S$ must be the type $S=a^{xn-1}b^{yn-1}$, where $x+y=k+1$ and $\langle ab^{-1}\rangle=C_n$. Moreover $\prod_{kn-2}(S)=C_n.$

\end{enumerate}
\end{lemma}

\begin{lemma}\label{lemma5}{\rm(\cite{GLi})}
Let $G$ be a non-cyclic finite solvable group of order $n$. Then every sequence over $G$ of length $kn+\frac{3}{4}n-1$ contains a product-one subsequence of length $kn$.
\end{lemma}

Let $\phi(n)=|\{k\text{ }|\text{ }1\le k\le n-1, (k,n)=1\}|$ be the Euler $\phi$ function.

\begin{lemma} \label{lemma6}
Let $G\cong C_m\ltimes_{\varphi} C_m$ and $p$ is the largest prime divisor of $m$, then $G$ contains a normal subgroup $N$, which is isomorphic to $C_p\times C_p$. Moreover
$$G/N\cong C_{\frac{m}{p}}\ltimes_{\varphi} C_{\frac{m}{p}}.$$
\end{lemma}
\begin{proof} Let $G=\langle x,y\rangle$ with ord$(x)=m$ and ord$(y)=m$, $yxy^{-1}=x^s$, $\langle x\rangle\cap\langle y\rangle={1}$ and $m=np^k$, where $(n,p)=1$ and $(m,s)=1$. We set $N=\langle x^{np^{k-1}},y^{np^{k-1}}\rangle$. We will prove that $N\cong C_p\times C_p$ and $N$ is a normal subgroup of $G$.

Firstly, we check that $N\cong C_p\times C_p$. It can be easily deduced from the fact that $(x^{np^{k-1}})^p=(y^{np^{k-1}})^p=1$, and that $N$ is not generated by $x^{np^{k-1}}$ or $y^{np^{k-1}}$, otherwise $y^{np^{k-1}}=x^{np^{k-1}t}$, where $1\leq t\leq p-1$ and since $\langle x\rangle\cap\langle y\rangle={1}$, this will lead to $y^{np^{k-1}}=1$, together with ord$(y)=m$, a contradiction.

Secondly, we prove that $N$ is normal. We just need to check that $yx^{np^{k-1}}y^{-1}\in N$ and $x^{-1}y^{np^{k-1}}x\in N$. It is clear that $yx^{np^{k-1}}y^{-1}=x^{snp^{k-1}}\in N$ and $x^{-1}y^{np^{k-1}}x=x^{s^{np^{k-1}}-1}y^{np^{k-1}}$. Now we claim that $s^{np^{k-1}}\equiv1\pmod{np^k}$. Since $s^{np^{k}}\equiv 1 \pmod{np^k}$ and $s^{\phi(np^k)}\equiv 1 \pmod{np^k}$ it follows that $s^{(np^{k},\phi(np^k))}\equiv 1 \pmod{np^k}$. Moreover, since $p$ is the largest prime divisor of $np^k$, simple calculation shows that $(np^k, \phi(np^k)) | np^{k-1}$, this proves the claim and it follows that $x^{-1}y^{np^{k-1}}x=x^{s^{np^{k-1}}-1}y^{np^{k-1}}=y^{np^{k-1}}\in N$.

Moreover, let $\overline{x}=xN$, $\overline{y}=yN$, the structure of the group $G/N$ is $\langle \overline{x},\overline{y}\rangle$ with $\overline{x}^{np^{k-1}}=\overline{1}=\overline{y}^{np^{k-1}}$ and $\overline{y} \overline{x}(\overline{y})^{-1}=\overline{x}^s$.
\end{proof}

We shall use the following definitions:

\begin{enumerate}

\item $Q_{2^{n+1}}=\langle a,b\text{ }|\text{ }a^{2^n}=1,\text{ }b^2=a^{2^{n-1}},\text{ }b^{-1}ab=a^{-1}\rangle$, the generalized quaternion group, where $n>1$.

\item $SD_{2^{n+1}}=\langle a,b\text{ }|\text{ }a^{2^n}=b^2=1,\text{ }bab=a^{-1+2^{n-1}}\rangle$, the semi-dihedral group, where $n>2$.

\item $Dic_{4n}=\langle a,b\text{ }|\text{ }a^{2^n}=1,\text{ }b^2=a^n,\text{ }bab^{-1}=a^{-1}\rangle$, the dicyclic group, where $n>1$.

\end{enumerate}

\begin{lemma} \label{pgroup}
Let $G$ be a non-cyclic $p$-group where $p$ is a prime, then $G$ contains a normal subgroup $H$ such that
$$H\cong C_p\times C_p,$$
unless $p=2$
and $G$ is a dihedral, semi-dihedral, or generalized quaternion group.
\end{lemma}
\begin{proof}
See Lemma 1.4 in \cite{Ber}.
\end{proof}

\begin{lemma} \label{nilpotent}{\rm(\cite{Group})}
Let $G$ be a finite nilpotent group, then $G \cong \prod_p G_p$, where $p$
ranges over all the prime divisors of $|G|$ and $G_p$ is the Sylow $p$-subgroup of $G$.
\end{lemma}

\section{On $\eta(G)$, $\mathsf s(G)$ and $\mathsf s_{mn\mathbb N}(G)$}

In this section, we are going to investigate $\eta(G)$, $\mathsf s(G)$ and $\mathsf s_{mn\mathbb N}(G)$ for the group $G\cong C_m\ltimes_{\varphi} C_{mn}$.

\begin{remark}
Let $G$ and $H$ be finite groups with a homomorphism $G\stackrel{\theta}{\rightarrow}H$ and $\ker\theta \cong N$. Let $S=g_1\bdot \ldots \bdot g_n$ be a sequence over $G$, then $\theta(S)=\theta(g_1)\bdot \ldots \bdot \theta(g_n)$ is a sequence over $H$. If $\theta(S)$ has a product-one subsequence $\theta(T)$, where $T=g_{i_1}\bdot \ldots \bdot g_{i_m}$, then by our notation, this means that $1\in\pi(\theta(T))$. In the following, for convenience of our calculation, without loss of generality, we may assume that $1=\sigma(\theta(T))$, equivalently $\sigma(T)\in\ker\theta\cong N$.
\end{remark}

\begin{lemma} \label{lemma7}
We have the following:
\begin{enumerate}

\item Let $G\cong C_m\ltimes_{\varphi} C_{mn}$, then any sequence $S$ over $G$ of length $|S|=2m+mn-2$ contains a product-one subsequence $T$ of length $|T|\in[1,mn]$;

\item Let $G\cong C_m\ltimes_{\varphi} C_{mn}$, then any sequence $S$ over $G$ of length $|S|=2m+2mn-3$ contains a product-one subsequence $T$ of length $|T|=mn$;

\item Let $G\cong C_m\ltimes_{\varphi} C_{mn}$, then any sequence $S$ over $G$ of length $|S|=m+2mn-2$ contains a product-one subsequence $T$ of length $|T|\equiv0\pmod{mn}$.
\end{enumerate}
\end{lemma}

\begin{proof}
The proofs of (1) and (2) are very similar, for simplicity of statement, we shall prove them simultaneously. We divide the proof into two steps:

{\bf Step 1.} We first prove the result for the case $C_m\ltimes_{\varphi} C_m$ and we proceed by induction on the number of prime divisors of $m$.

If $m=p$ is a prime, since $C_p\ltimes_{\varphi} C_p=C_p\times C_p$, then by Lemma \ref{lemma2}.(3) we get the desired result for both (1) and (2).

If $m$ is not a prime, we can write $m=tp$, where $p$ is the largest prime divisor of $m$. Then by Lemma \ref{lemma6} we have a homomorphism $C_m\ltimes_{\varphi} C_m\stackrel{\theta}{\rightarrow}C_t\ltimes_{\varphi} C_t$ and $\ker\theta \cong C_p\times C_p$.

Let $S=g_1\bdot \ldots \bdot g_{3m-2}$ (resp. $S=g_1\bdot \ldots \bdot g_{4m-3}$) be any sequence over $C_m\ltimes_{\varphi} C_m$ and $\theta(S)=\theta(g_1)\bdot \ldots \bdot \theta(g_{3m-2})$ (resp. $\theta(S)=\theta(g_1)\bdot \ldots \bdot \theta(g_{4m-3})$) be a sequence over $C_t\ltimes_{\varphi} C_t$. Since $|S|\geq 3t-2$ (resp. $|S|\geq 4t-3$), by induction, there exists a subsequence $T_1$ of $S$ such that $1\le|T_1|\leq t$ (resp. $|T_1|=t$) and $1\in\pi(\theta(T_1))$. By the above remark, we may assume that $1=\sigma(\theta(T_1))$ and therefore $\sigma(T_1)\in \ker\theta$.

If $|S{T_1}^{-1}|\geq 3t-2$ (resp. $|S{T_1}^{-1}|\geq 4t-3$), then there exists a subsequence $T_2$ of $S{T_1}^{-1}$ such that $1\le|T_2|\leq t$ (resp. $|T_2|=t$) and $\sigma(T_2)\in \ker\theta$, continuing this process, we can find disjoint subsequences $T_1,\ldots,T_r$ such that $1\le|T_i|\leq t$ (resp. $|T_i|=t$) and $\sigma(T_i)\in \ker\theta$ until $|S{T_1}^{-1}\bdot\bdot\bdot{T_r}^{-1}|\leq 3t-3$ (resp. $|S{T_1}^{-1}\bdot\bdot\bdot{T_r}^{-1}|\leq 4t-4$). Then $r\geq \lceil\frac{3m-2-3t+3}{t}\rceil=3p-2$ (resp. $r\geq \lceil\frac{4m-3-4t+4}{t}\rceil=4p-3$). Since $\sigma(T_i)\in \ker\theta\cong C_p\times C_p$ for $i\in [1,r]$, together with Lemma \ref{lemma2} that $\eta(C_p\times C_p)=3p-2$ (resp. $\mathsf s(C_p\times C_p)=4p-3$), there exists $T_{i_1},\ldots,T_{i_k}$ such that $1\leq k\leq p$ and $\sigma(T_{i_1})\bdot \ldots \bdot \sigma(T_{i_k})$ (resp. there exists $T_{i_1},\ldots,T_{i_p}$ such that $\sigma(T_{i_1})\bdot \ldots \bdot \sigma(T_{i_p})$) is a product-one sequence over $C_p\times C_p$. Thus $T=T_{i_1}\bdot \ldots \bdot T_{i_k}$ (resp. $T=T_{i_1}\bdot \ldots \bdot T_{i_p}$) is a product-one subsequence over $G$ of length $|T|=|T_{i_1}|+\cdots+|T_{i_k}|\leq tp=m$ (resp. $|T|=|T_{i_1}|+\cdots+|T_{i_p}|=tp=m$ ). This completes the proof of the case $C_m\ltimes_{\varphi} C_m$.

{\bf Step 2.} For the case $G\cong C_m\ltimes_{\varphi} C_{mn}$, we have the following homomorphism $$C_m\ltimes_{\varphi} C_{mn}\stackrel{\psi}{\rightarrow}C_m\ltimes_{\varphi} C_m$$ with $\ker\psi=C_n$. Then let $S=g_1\bdot \ldots \bdot g_{2m+mn-2}$ (resp. $S=g_1\bdot \ldots \bdot g_{2m+2mn-3}$) be any sequence over $C_m\ltimes_{\varphi} C_{mn}$ and $\psi(S)=\psi(g_1)\bdot \ldots \bdot \psi(g_{2m+mn-2})$ (resp. $\psi(S)=\psi(g_1)\bdot \ldots \bdot \psi(g_{2m+2mn-3})$) be a sequence over $C_m\ltimes_{\varphi} C_m$. Since $|S|\geq 3m-2$ (resp. $|S|\geq 4m-3$), by the above case, there exists a subsequence $T_1$ of $S$ such that $1\le|T_1|\leq m$ (resp. $|T_1|=m$) and $\sigma(T_1)\in \ker\psi$. If $|S{T_1}^{-1}|\geq 3m-2$ (resp. $|S{T_1}^{-1}|\geq 4m-3$), then there exists a subsequence $T_2$ of $S{T_1}^{-1}$ such that $1\le|T_2|\leq m$ (resp. $|T_2|=m$) and $\sigma(T_2)\in \ker\psi$, continuing this process, we can find disjoint subsequences $T_1,\ldots,T_r$ such that $1\le|T_i|\leq m$ (resp. $|T_i|=m$) and $\sigma(T_i)\in \ker\psi$ until $|S{T_1}^{-1}\bdot\bdot\bdot{T_r}^{-1}|\leq 3m-3$ (resp. $|S{T_1}^{-1}\bdot\bdot\bdot{T_r}^{-1}|\leq 4m-4$). Then $r\geq \lceil\frac{2m+mn-2-3m+3}{m}\rceil=n$ (resp. $r\geq \lceil\frac{2m+2mn-3-4m+4}{m}\rceil=2n-1$). Since $\sigma(T_i)\in \ker\psi\cong C_n$ for $i\in [1,r]$ and $\eta(C_n)=n$ (resp. $\mathsf s(C_n)=2n-1$) by Lemma \ref{lemma2}, there exists $T_{i_1},\ldots,T_{i_k}$ such that $1\leq k\leq n$ and $\sigma(T_{i_1})\bdot \ldots \bdot \sigma(T_{i_k})$ (resp. there exists $T_{i_1},\ldots,T_{i_n}$ such that $\sigma(T_{i_1})\bdot \ldots \bdot \sigma(T_{i_n})$) is a product-one sequence over $C_n$. Thus $T=T_{i_1}\bdot \ldots \bdot T_{i_k}$ (resp. $T=T_{i_1}\bdot \ldots \bdot T_{i_n}$) is product-one subsequence over $G$ of length $|T|=|T_{i_1}|+\cdots+|T_{i_k}|\leq mn$ (resp. $|T|=|T_{i_1}|+\cdots+|T_{i_n}|=mn$). This completes the proof of both (1) and (2).

{\bf (3)} The proof of (3) also follows the same two steps as above, meanwhile it makes use of the above result.

{\bf Step 1.} We first prove the result for the case $C_m\ltimes_{\varphi} C_m$ and we proceed by induction on the number of prime divisors of $m$.

If $m=p$ is a prime, then since $C_p\ltimes_{\varphi} C_p=C_p\times C_p$, we get the desired result from Lemma \ref{lemma2}.

If $m$ is not a prime, we can write $m=tp$, where $p$ is the largest prime divisor of $m$. Then by Lemma \ref{lemma6} we have a homomorphism $C_m\ltimes_{\varphi} C_m\stackrel{\theta}{\rightarrow}C_t\ltimes_{\varphi} C_t$ and $\ker\theta \cong C_p\times C_p$.

Let $S=g_1\bdot \ldots \bdot g_{3m-2}$ be any sequence over $C_m\ltimes_{\varphi} C_m$. We set $H=G\times C_{m}=G\times \langle e\rangle $. We can construct the following homomorphism
$$(C_m\ltimes_{\varphi} C_m)\times C_m\stackrel{\psi}{\rightarrow}(C_t\ltimes_{\varphi} C_t)\times C_t$$
$$(g,e^h)\mapsto(\theta(g),e^{hp})$$
where $g\in C_m\ltimes C_m$, $e^h\in C_m$.
It can be easily checked that $\ker\psi\cong C_p\times C_p\times C_p$. Set $S^H=g_1e\bdot \ldots \bdot g_{3m-2}e$, thus it suffices to prove that $S^H$ has a non-empty product-one subsequence. We have $\theta(S)=\theta(g_1)\bdot \ldots \bdot \theta(g_{3m-2})$ be a sequence over $C_t\ltimes_{\varphi} C_t$. Since $|\theta(S)|\geq 4t-3$, there exists by part (2) that we have proved above a subsequence $T_1$ of $S$ such that $|T_1|=t$ and $\sigma(T_1)\in \ker\theta$, if $|ST_1^{-1}|\geq 4t-3$, then there exists a subsequence $T_2$ of $ST_1^{-1}$ such that $|T_2|=t$ and $\sigma(T_2)\in \ker\theta$, continuing this process, we can find disjoint subsequences $T_1,\ldots,T_r$ such that $|T_i|=t$ and $\sigma(T_i)\in \ker\theta$ until $|ST_1^{-1}\bdot\bdot\bdot{T_r^{-1}}|\leq 4t-4$. Then $r\geq \lceil\frac{3m-2-4t+4}{t}\rceil=3p-3$ and at least we have $T_1,\ldots,T_{3p-3}$ such that $|T_i|=t$ and $\sigma(T_i)\in \ker\theta$ for $1\le i\le 3p-3$. Since $|ST_1^{-1}\cdots{T_{3p-3}^{-1}}|=3m-2-(3p-3)t=3t-2$, and by induction, $ST_1^{-1}\bdot \bdot \bdot{T_{3p-3}^{-1}}$ contains a product-one subsequence $J$ of length $|J|\equiv0\pmod{t}$ and $\sigma(J)\in \ker\theta$, by renumbering the indices, we denote $J=T_{3p-2}$. We set $W=T_1^H\bdot \ldots \bdot T_{3p-2}^H$ (the definition of $T_i^H$ is the same as $S^H$), then $U=\sigma(T_1^H)\bdot \ldots \bdot \sigma(T_{3p-2}^H)$ is a sequence over $\ker\psi\cong C_p\times C_p\times C_p$ of length $|U|=3p-2$. By Lemma \ref{lemma2} stating that $\mathsf D(C_p\times C_p\times C_p)=3p-2$, we can find a product-one subsequence $V=\sigma(T_{i_1}^H)\bdot \ldots \bdot \sigma(T_{i_k}^H)$ of $U$ for some $k\ge 1$ and $T^H=T_{i_1}^H\bdot \ldots \bdot T_{i_k}^H$ is a product-one subsequence of $S^H$. Therefore, by our construction, $T=T_{i_1}\bdot \ldots \bdot T_{i_k}$ is a product-one subsequence of $S$ such that the length of $T$ satisfies $|T|\equiv0\pmod{m}$. This completes the proof of the case $C_m\ltimes_{\varphi} C_m$.

{\bf Step 2.} For the case $G\cong C_m\ltimes_{\varphi} C_{mn}$, we have the following homomorphism $$C_m\ltimes_{\varphi} C_{mn}\stackrel{\theta}{\rightarrow}C_m\ltimes_{\varphi} C_m$$ with $\ker\theta\cong C_n$.
Let $S=g_1\bdot \ldots \bdot g_{m+2mn-2}$ be any sequence of length $m+2mn-2$ over $G$. We set $H=G\times C_{mn}=G\times \langle e\rangle $. We can construct the following homomorphism
$$(C_m\ltimes_{\varphi} C_{mn})\times C_{mn}\stackrel{\psi}{\rightarrow}(C_m\ltimes_{\varphi} C_m)\times C_m$$
$$(g,e^h)\mapsto(\theta(g),e^{hn})$$
where $g\in C_m\ltimes C_{mn}$, $e^h\in C_{mn}$.
It can be easily checked that $\ker\psi\cong C_n\times C_n$. Set $S^H=g_1e\bdot \ldots \bdot g_{m+2mn-2}e$, thus it suffices to prove that $S^H$ has a non-empty product-one subsequence. We have $\theta(S)=\theta(g_1)\bdot \ldots \bdot \theta(g_{m+2mn-2})$ be a sequence over $C_m\ltimes_{\varphi} C_m$. Since $|\theta(S)|\geq 4m-3$,  there exists by part (2) that we have proved above a subsequence $T_1$ of $S$ such that $|T_1|=m$ and $\sigma(T_1)\in \ker\theta$, if $|ST_1^{-1}|\geq 4m-3$, then there exists a subsequence $T_2$ of $ST_1^{-1}$ such that $|T_2|=m$ and $\sigma(T_2)\in \ker\theta$, continuing this process, we can find disjoint subsequences $T_1,\ldots,T_r$ such that $|T_i|=m$ and $\sigma(T_i)\in \ker\theta$ until $|ST_1^{-1}\bdot\bdot\bdot{T_r^{-1}}|\leq 4m-4$. Then $r\geq \lceil\frac{m+2mn-2-4m+4}{m}\rceil=2n-2$ and at least we have $T_1,\ldots,T_{2n-2}$ such that $|T_i|=n$ and $\sigma(T_i)\in \ker\theta$ for $1\le i\le 2n-2$. Since $|ST_1^{-1}\cdots{T_{2n-2}^{-1}}|=m+2mn-2-(2n-2)m=3m-2$, and by the above case, $ST_1^{-1}\bdot \bdot \bdot{T_{2n-2}^{-1}}$ contains a product-one subsequence $J$ of length $|J|\equiv0\pmod{m}$ and $\sigma(J)\in \ker\theta$, by renumbering the indices, we denote $J=T_{2n-1}$. We set $W=T_1^H\bdot \ldots \bdot T_{2n-1}^H$, then $U=\sigma(T_1^H)\bdot \ldots \bdot \sigma(T_{2n-1}^H)$ is a sequence over $\ker\psi\cong C_n\times C_n$ of length $|U|=2n-1$. By Lemma \ref{lemma2} stating that $\mathsf D(C_n\times C_n)=2n-1$, we can find a product-one subsequence $V=\sigma(T_{i_1}^H)\bdot \ldots \bdot \sigma(T_{i_k}^H)$ of $U$ for some $k\ge 1$ and $T^H=T_{i_1}^H\bdot \ldots \bdot T_{i_k}^H$ is a product-one subsequence of $S^H$. Therefore, by our construction, $T=T_{i_1}\bdot \ldots \bdot T_{i_k}$ is a product-one subsequence of $S$ such that the length of $T$ satisfies $|T|\equiv0\pmod{mn}$. This completes the proof.
\end{proof}

{\sl Proof of Theorem \ref{theorem2}}.
Firstly, it follows from Lemma \ref{lemma7} that $\mathsf s_{mn\mathbb N}(G)\le m+2mn-2$. Let $S=x^{[m-1]}\bdot y^{[mn-1]}\bdot 1^{[mn-1]}$ be a sequence of length $m+2mn-3$ over $G$. We can easily check that $S$ contains no product-one subsequence $T$ of length $|T|\equiv 0\pmod{mn}$. Indeed, since $x^uy^v=y^{s^uv}x^u$ for $u\ge 0$, $v\ge 0$, that means the power of $x$ does not change and the subsequence $x^{[m-1]}\bdot y^{[mn-1]}$ contains no product-one subsequence. Therefore $\mathsf s_{mn\mathbb N}(G)\ge m+2mn-2$.

Next, obviously, if $\mathsf e(G)=mn$, then it follows from Lemma \ref{lemma7} that $\eta(G)\le 2m+mn-2$ and $\mathsf s(G)\le 2m+2mn-3$.
\qed

\section{On the Erd\H{o}s-Ginzburg-Ziv theorem}

In this section, we will prove Theorem \ref{theorem1}.

{\sl Proof of Theorem \ref{theorem1}}.
If $m=1$, then the desired result follows from Lemma \ref{lemma1}. So for the rest we may assume that $m\ge 2$.

Let $S$ be a sequence over $G$ of length $|G|+\frac{|G|}{m}+m-2$, we will show that $S$ has a product-one subsequence of length $|G|$. Let $\theta$ be the homomorphism
$$\theta:G\rightarrow G/N\cong C_m\ltimes_{\varphi} C_m$$
where $\ker\theta\cong N$. Since $G/N\cong C_m\ltimes_{\varphi} C_m$, and by Lemma \ref{lemma7}.(2), we can repeatedly remove the product-one subsequences from $\theta(S)$ of length $m$ until $|\theta(S)|\le 4m-4$. In other words, we obtain a factorization $S=S_1\bdot\ldots\bdot S_rS'$ with
$$|S_i|=m\text{ and }\sigma(S_i)\in \ker\theta\text{ for }1\le i\le r,\text{ and }|S'|\le 4m-4.$$
Consequently, $$r\ge\lceil\frac{|G|+\frac{|G|}{m}+m-2-4m+4}{m}\rceil=
\frac{|G|}{m}+\frac{|G|}{m^2}-2.$$
If $N$ is not a cyclic subgroup, then by Lemma \ref{lemma5} and the fact that $\frac{|G|}{m}
+\frac{|G|}{m^2}-2\ge m\frac{|G|}{m^2}
+\frac{3}{4}\frac{|G|}{m^2}-1$, the sequence $\sigma(S_1)\bdot\ldots\bdot\sigma(S_{\frac{|G|}{m}
+\frac{|G|}{m^2}-2})$ contains a product-one subsequence of length $\frac{|G|}{m}$, and the proof follows.

Otherwise, we may assume that $N$ is a cyclic subgroup of $G$. Let $T=SS_1^{-1}\bdot \ldots \bdot S_{\frac{|G|}{m}
+\frac{|G|}{m^2}-2}^{-1}$, then $|T|=3m-2$ and $\theta(T)$ contains a product-one subsequence of length $m$ or $2m$ in $C_m\ltimes_{\varphi} C_m$ by Lemma \ref{lemma7}.(3). We distinguish two cases.

{\bf Case 1.} If $T$ contains a subsequence $S_{\frac{|G|}{m}
+\frac{|G|}{m^2}-1}$ of length $m$ with $\sigma(S_{\frac{|G|}{m}
+\frac{|G|}{m^2}-1})\in \ker\theta$.

By Lemma \ref{lemma4}(1), the sequence $\sigma(S_1)\bdot \ldots \bdot\sigma(S_{\frac{|G|}{m}
+\frac{|G|}{m^2}-1})$ over $N$ contains a product-one subsequence of length $\frac{|G|}{m}$. By rearrangement we may assume that $\sigma(S_1)\bdot \ldots \bdot\sigma(S_{\frac{|G|}{m}})=1$ then $J=S_1\bdot \ldots \bdot S_{\frac{|G|}{m}}$ is a product-one subsequence over $G$ of length $|G|$.

{\bf Case 2.} $T$ contains no subsequence $T'$ of length $m$ with $\sigma(T')\in \ker\theta$.

Therefore $T$ contains a subsequence $J$ of length $2m$ with $\sigma(J)\in \ker\theta$. Let $W=\sigma(S_1)\bdot \ldots \bdot\sigma(S_{\frac{|G|}{m}
+\frac{|G|}{m^2}-2})$, then $W$ is a sequence of length $\frac{|G|}{m}
+\frac{|G|}{m^2}-2$ over $C_{\frac{|G|}{m^2}}$. If $W$ contains a product-one subsequence of length $\frac{|G|}{m}$, then we are done.

Otherwise, from Lemma \ref{lemma4}.(2) it follows that $\prod_{\frac{|G|}{m}-2}(W)=C_{\frac{|G|}{m^2}}$, thus $\sigma(J)^{-1}\in \prod_{\frac{|G|}{m}-2}(W)$ and $\sigma(S_{i_1})\bdot \ldots \bdot\sigma(S_{i_{\frac{|G|}{m}-2}})\sigma(J)=1$ for $1\le i_1<\cdots<i_{\frac{|G|}{m}-2}\le\frac{|G|}{m}
+\frac{|G|}{m^2}-2$. Hence $S_{i_1}\bdot \ldots \bdot S_{i_{\frac{|G|}{m}-2}}J$ is a product-one subsequence of length $(\frac{|G|}{m}-2)m+2m=|G|$ over $G$. This completes the proof of the first part of the theorem.

In particular, if $G\cong C_m\ltimes_{\varphi} C_{mn}$, where $m,n$ are positive integers, then we may assume that $G$ is generated by two elements $x,y$ such that $\langle x\rangle\cap\langle y\rangle=1$, where ord$(x)=m$, ord$(y)=mn$ and $xyx^{-1}=y^s$ with $1\le s\le mn-1$. Note that $G/\langle x^m\rangle\cong C_m\ltimes_{\varphi} C_m$.

Let $S=x^{[m-1]}\bdot y^{[mn-1]}$ be a sequence over $G$. We can easily check that $S$ contains no product-one subsequence, since $x^uy^v=y^{s^uv}x^u$ for $u\ge 0$, $v\ge 0$, that means the power of $x$ does not change. Then by Lemma \ref{lemma1} and the above result we get
$$nm^2+nm+m-2\le nm^2+\mathsf d(G)\le \mathsf E(G)\le nm^2+nm+m-2,$$
which completes the proof.
\qed

\section{Noether number for $C_m\ltimes_{\varphi} C_{mn}$}

In this section, we shall determine the Noether number for $C_m\ltimes_{\varphi} C_{mn}$. Firstly, we recall the detailed definition of the Noether number.

Let $G$ be a finite group and $\rho: G\rightarrow GL(V)$ be a finite dimensional linear representation of $G$ over a field $\mathbb{F}$ of characteristic which is not dividing the group order $|G|$. Let $\mathbb{F}[V]$ denote the graded algebra of polynomial functions on $V$. We can regard $\mathbb{F}[V]$ as the symmetric algebra on $V^{*}$, the dual space of $V$. In other words, if $z_1,\ldots,z_n\in V^*$ is a basis, then $\mathbb{F}[V]$ is just the polynomial ring $\mathbb{F}[z_1,\ldots,z_n]$. The elements in $\mathbb{F}[V]$ are the homogeneous polynomials in the linear forms $z_1,\ldots,z_n$ with coefficient in $\mathbb F$. Since $G$ has a natural action on $V$, i.e., through the representation $\rho$, thus we can view $V$ as a $G$-module. Moreover we can induce a right action of $G$ on $V^*$ as follows: $$x^g(v)=x(g\cdot v)=x(\rho(g)v).$$
Therefore, this action can be naturally extended to an action on $\mathbb{F}[V]$. The central topic of invariant theory is to study the algebra of polynomial invariants which is defined as follows:
$$\mathbb{F}[V]^G=\{f\in\mathbb{F}[V]\text{ }|\text{ }f^g=f,\text{ for all }g\in G\}.$$

In the following, we assume that all the representations of $G$ are over a fixed base field $\mathbb F$ of characteristic which is not dividing the group order $|G|$. Now we can give the definition of the Noether number $\beta(G)$.

\begin{definition}
We define
$$\beta(G,V):=\min\{s\in\mathbb N\text{ }|\text{ }\mathbb{F}[V]^G\text{ is generated by invariants of degree }\le s\},$$
and
$$\beta(G):=\max\{\beta(G,V)\text{ }|\text{ }\rho:G\rightarrow GL(V)\text{ is a finite dimensional representation}\}.$$
\end{definition}

Also, we recall the following generalizations of Noether number and Davenport constant.

\begin{definition}{\rm(\cite{CD2})}
We define
$$\beta_k(G,V)=\min\{s\in \mathbb N\text{ }|\text{ }\mathbb F[V]^G \text{ is generated as an }(\mathbb F[V]^G)^k\text{-algebra by }\mathbb F[V]^G_{\le s}\},$$
and
$$\beta_k(G):=\max\{\beta_k(G,V)\text{ }|\text{ }\rho:G\rightarrow GL(V)\text{ is a finite dimensional representation}\}.$$
\end{definition}
Here $\beta_k(G)$ is called the $generalized$ $Noether$ $number$ of $G$.

\begin{definition}{\rm(\cite{HK})}
Let $G$ be a finite abelian group, we define $\mathsf D_k(G)$ to be the smallest integer $t$ such that any sequence $S$ over $G$ of length $|S|\ge t$ has $k$ non-empty disjoint product-one subsequences $T_1,\cdots,T_k$.
\end{definition}
Here $\mathsf D_k(G)$ is called the $generalized$ $Davenport$ $constant$ of $G$.

\begin{lemma}\label{Dk}
We have the following:
\begin{enumerate}

\item Let $G$ be a finite abelian group, then $\beta_k(G)=\mathsf D_k(G)$;

\item Let $p$ be a prime and $G=C_p\times C_p$, then $\mathsf D_k(G)=kp+p-1$.

\end{enumerate}
\end{lemma}

\begin{proof}
(1) See Proposition 4.7.4, \cite{CDG}. (2) See Theorem 6.1.5, \cite{GH}
\end{proof}

We will employ the following crucial reduction lemma due to Cziszter and Domokos, see \cite{CD2} or \cite{CD3}.

\begin{lemma}\label{lemma8}{\rm(\cite{CD2})}
Let $H$ be a subgroup of $G$ and $V$ a $G$-module.
\begin{enumerate}
\item We have $\beta_k(\mathbb{F}[V]_{+},\mathbb{F}[V]^G)\le
    \beta_{k[G:H]}(\mathbb[V]_{+},\mathbb{F}[V]^H).$ In particular,
    $$\beta_k(G,V)\le\beta_{k[G:H]}(H,V);$$

\item If $H$ is normal in $G$, then $\beta_k(G,V)\le\beta_{\beta_k(G/H)}(H,V)$;

\item Let $H$ be a normal subgroup of a finite group $G$ with $G/H$ abelian. Then for all positive integers $k$ we have the inequality $$\beta_k(G)\ge\beta_k(H)+\mathsf D(G/H)-1.$$
\end{enumerate}
\end{lemma}

With the help of the above lemma, we may use the induction method to prove our result.

\begin{lemma}\label{lemma9}
Let $V$ be any $G$-module, then $\beta_k(C_m\ltimes_{\varphi} C_m,V)\le km+m-1.$
\end{lemma}

\begin{proof}
Let $m=p_1p_2\cdots p_l$ be the decomposition of $m$ into prime numbers, where $p_1\ge p_2\ge \cdots\ge p_l$. We shall prove by induction on the number of prime divisors of $m$. If $l=1$, then by Lemma \ref{Dk}, we have $\beta_k(C_m\ltimes_{\varphi} C_m,V)=\beta_k(C_{p_1}\times C_{p_1},V)\le\beta_k(C_{p_1}\times C_{p_1})=\mathsf D_k(C_{p_1}\times C_{p_1})=kp_1+p_1-1=km+m-1.$ Assume that the claim has been proven for the case $l=t$. When $l=t+1$, by Lemma \ref{lemma6}, $G$ contains a normal subgroup $N$, which is isomorphic to $C_{p_1}\times C_{p_1}$. Moreover
$G/N\cong C_{\frac{m}{p_1}}\ltimes_{\varphi} C_{\frac{m}{p_1}}$. Then by induction, $\beta_k(C_{\frac{m}{p_1}}\ltimes_{\varphi} C_{\frac{m}{p_1}},V)\le\frac{km}{p_1}+\frac{m}{p_1}-1$. Finally, by Lemma \ref{lemma8}(2), we have $\beta_k(C_m\ltimes_{\varphi} C_m,V)\le\beta_{\beta_k(C_{\frac{m}{p_1}}\ltimes_{\varphi} C_{\frac{m}{p_1}})}(C_{p_1}\ltimes_{\varphi} C_{p_1},V)=\beta_{\frac{km}{p_1}+\frac{m}{p_1}-1}
(C_{p_1}\ltimes_{\varphi} C_{p_1},V)
\le(\frac{km}{p_1}+\frac{m}{p_1}-1)p_1+p_1-1=km+m-1$.
\end{proof}

{\sl Proof of Theorem \ref{theorem3}}.
Let $V$ be any $G$-module. We will show that $\beta(G,V)\le mn+m-1$. Because $C_m\ltimes_{\varphi} C_{mn}$ has a subgroup $H=C_m\ltimes_{\varphi} C_m$, by Lemma \ref{lemma8}(1), we have $\beta(G,V)\le\beta_{[G:H]}(H,V)=\beta_n(C_m\ltimes_{\varphi} C_m,V)$. Then by Lemma \ref{lemma9}, we have $\beta(G,V)\le mn+m-1$. Consequently, $$\beta(G)=\sup_V\beta(G,V)\le mn+m-1.$$

On the other hand, since $G$ has a normal subgroup $H=C_{mn}$ and $G/H\cong C_m$, by Lemma \ref{lemma8}(3), we have
$$\beta(G)\ge\beta(H)+\mathsf D(G/H)-1=mn+m-1.$$ Therefore $\beta(G)=mn+m-1$.
\qed

After comparing this with the result in Theorem \ref{theorem1}, we conclude that for $C_m\ltimes_{\varphi} C_{mn}$, we also have $\beta(G)=\mathsf d(G)+1,$ which is the same as the abelian group case.

\bigskip

{\sl Proof of Theorem \ref{theorem4}}.
Assume that $p_1<\cdots< p_n$ are all the distinct prime divisors of $|G|$. By Lemma \ref{nilpotent}, we have $G\cong \prod_{i=1}^n G_{p_i}$, where $G_{p_i}$ is the Sylow $p_i$-subgroup of $G$.
We distinguish two cases:

(1) If there exists $j$ such that $1\le j\le n$ and $G_{p_j}$ has a subgroup $H\cong C_{p_j}\times C_{p_j}$.

By Lemma \ref{Dk} and \ref{lemma8}(1), we have \begin{align*}
\beta(G,V)&\le\beta_{[G:H]}(H,V)=\beta_{\frac{|G|}{p_j^2}}(C_{p_j}\times C_{p_j},V)\\
&\le \beta_{\frac{|G|}{p_j^2}}(C_{p_j}\times C_{p_j})=\frac{|G|}{p_j}+p_j-1\le\frac{|G|}{p_1}+p_1-1,
\end{align*}
where the last step is just a simple calculation, and we are done.

(2) If $G_{p_j}$ has no subgroup $H\cong C_{p_j}\times C_{p_j}$ for any $j$ with $1\le j\le n$.

Because $G$ is non-cyclic, by Lemma \ref{pgroup} we must have $p_1=2$ and $G_{p_1}$ is a dihedral, semi-dihedral, or generalized quaternion group. In this case, by Lemma 1.4 in \cite{Ber} again $G_{p_1}$ contains a cyclic subgroup $C$ of index 2. But then $C\times G_{p_2}\times\cdots\times G_{p_n}$ is itself a cyclic subgroup of index 2 in $G$. By Theorem 10.3 in \cite{CD3}, we have for such groups $\beta(G)=\frac{1}{2}|G|+1=\frac{1}{2}|G|+2-1$, unless $G$ is the dicyclic group, in which case $\beta(G)=\frac{1}{2}|G|+2$. This proves the theorem.
\qed

\section{Concluding remarks}

A conjecture attributed to Pawale \cite{W} stated that $\beta(C_q\ltimes_{\varphi} C_p)=p+q-1$, where $p,q$ are primes such that $q|p-1$. This conjecture has been studied in \cite{C1}, \cite{CD2}, \cite{CD3} and \cite{DH}. In \cite[Example 5.2]{CDG}, the authors determined that $\beta(C_q\ltimes_{\varphi} C_{pq})=pq+q-1$. Based on our result, Theorem \ref{theorem3}, we have the following conjecture.
\begin{conjecture}
For any positive integer $n,m$
$$\beta(C_n\ltimes_{\varphi} C_m)=n+m-1.$$
\end{conjecture}

\subsection*{Acknowledgment}
The authors are very greatful to the referee for providing many valuable suggestions and pointing out some mistakes in a previous manuscript, in particular the current version of Theorem \ref{theorem4} is due to the referee's correction. The authors are also very greatful to Prof. Alfred Geroldinger for many helpful comments and corrections. A part of this work was done during a visit by H.B. Zhang to the University of Graz, he would like to show his gratitude to the host institution for very kind hospitality. D.C. Han was supported by the National Science Foundation of China Grant No.11601448 and the Fundamental Research Funds for the Central Universities Grant No.2682016CX121. H.B. Zhang was supported by the National Science Foundation of China Grant No.11671218 and China Postdoctoral Science Foundation Grant No. 2017M620936. At last, the authors want to thank their advisor Prof. Weidong Gao for continuous encouragement and support.

\bibliographystyle{amsplain}

\end{document}